\documentclass{amsart}
\usepackage[utf8]{inputenc}
\usepackage[polutonikogreek,greek,french,english]{babel}
\usepackage[scr=kp]{mathalpha}
\usepackage{dutchcal}
\usepackage[usenames,dvipsnames,svgnames,table]{xcolor}
\usepackage{mathtools}
\usepackage{amsmath}
\usepackage{amsthm}
\usepackage{geometry}
\usepackage{amssymb}
\usepackage{bbold}
\usepackage[all]{xy}
\usepackage{etoolbox}
\usepackage[hidelinks]{hyperref}
\usepackage{enumitem}
\usepackage[pdf]{pstricks}
\usepackage{pgfplots}
\pgfplotsset{compat = newest}
\usepackage{ytableau}
\makeindex
\usepackage{subcaption}

\usepackage{tikz}
\usepackage{tikz-cd}
\usetikzlibrary{math}
\usepackage{pgf}
\usetikzlibrary{arrows}
\usetikzlibrary{shapes.geometric}
\usetikzlibrary{shapes.misc}

\theoremstyle{plain}
\newtheorem{tho}[subsubsection]{Theorem}
\newtheorem{thintro}[subsection]{Theorem}
\newtheorem{lemme}[subsubsection]{Lemma}
\newtheorem{prop}[subsubsection]{Proposition}
\newtheorem{propintro}[subsection]{Proposition}
\newtheorem{cor}[subsubsection]{Corollary}

\theoremstyle{remark}
\newtheorem{rmq}[subsubsection]{Remark}

\theoremstyle{definition}
\newtheorem{deff}[subsubsection]{Definition}
\newtheorem{defintro}[subsection]{Definition}
\newtheorem{example}[subsubsection]{Example}

\newcommand{\cA}{\mathcal{A}}
\newcommand{\cE}{\mathcal{E}}
\newcommand{\cF}{\mathcal{F}}
\newcommand{\cG}{\mathcal{G}}
\newcommand{\cH}{\mathcal{H}}
\newcommand{\cI}{\mathcal{I}}
\newcommand{\cL}{\mathcal{L}}
\newcommand{\cml}{\mathcal l}
\newcommand{\cM}{\mathscr{M}}
\newcommand{\cN}{\mathscr{N}}
\newcommand{\cO}{\mathcal{O}}
\newcommand{\cQ}{\mathcal{Q}}

\newcommand{\cT}{\mathcal{T}}

\newcommand{\cV}{\mathcal{V}}
\newcommand{\cW}{\mathcal{W}}

\newcommand{\eA}{\mathbb{A}}
\newcommand{\eC}{\mathbb{C}}
\newcommand{\eE}{\mathbb{E}}
\newcommand{\eF}{\mathbb{F}}
\newcommand{\ek}{\mathbb{k}}
\newcommand{\eP}{\mathbb{P}}
\newcommand{\eR}{\mathbb{R}}
\newcommand{\eS}{\mathbb{S}}
\newcommand{\eV}{\mathbb{V}}
\newcommand{\eZ}{\mathbb{Z}}
\newcommand{\eUn}{\mathbb{1}}

\newcommand{\fC}{\mathfrak{C}}
\newcommand{\fH}{\mathfrak{H}}
\newcommand{\fp}{\mathfrak{p}}
\newcommand{\fS}{\mathfrak{S}}
\newcommand{\fU}{\mathfrak{U}}

\newcommand{\bfW}{\mathbf{W}}

\newcommand{\Thom}[2][]{\operatorname{Th}_{#1}\left(#2\right)}
\newcommand{\ulHom}[2][]{\underline{\operatorname{Hom}}_{#1}\left(#2\right)}

\newcommand{\Au}{{\eA^1}}
\newcommand{\GW}{GW}
\newcommand{\fId}{{\mathbf{I}}}
\newcommand{\Witt}{\operatorname{W}}
\newcommand{\KMW}[2][\ast]{K^{MW}_{#1}\left( #2 \right)}
\newcommand{\sKW}{\underline{K}^W}
\newcommand{\Chowtilde}[2][]{\widetilde{\operatorname{CH}}{}^{#1}\!\left(#2\right)}
\newcommand{\Chowtildecov}[2][]{\widetilde{\operatorname{CH}}_{#1}\left(#2\right)}
\newcommand{\StableHo}[2][]{\operatorname{S}\!\cH^{#1}\!\left(#2\right)}

\newcommand{\simquad}{\sim_q}

\newcommand{\Aut}{\mathrm{Aut}}
\newcommand{\Picard}[1]{\mathrm{Pic}\left(#1\right)}
\newcommand{\qscal}[1]{\left\langle #1 \right\rangle}
\newcommand{\Id}{\mathrm{Id}}
\newcommand{\SL}{\operatorname{SL}}
\newcommand{\Gl}{\operatorname{GL}}
\newcommand{\PGL}{\operatorname{PGL}}
\newcommand{\isomr}{\Tilde{\to}}
\newcommand{\Galois}{\operatorname{Gal
}}

\newcommand{\Spec}[1]{\operatorname{Spec}\left(#1\right)}
\newcommand{\Proj}[1]{\operatorname{Proj}\left(#1\right)}
\newcommand{\Symalg}[2][]{\operatorname{Sym}^{#1}\left(#2\right)}
\newcommand{\QCoh}[2][]{\operatorname{QCoh}^{#1}\left(#2\right)}

\newcommand{\Grass}{\operatorname{Gr}}
\newcommand{\Schcat}{\mathbf{Sch}}
\newcommand{\Hilbsch}[2][]{\mathcal{H}ilb_{#1}\left(#2\right)}
\newcommand{\Hdrtes}{\Sigma_1}
\newcommand{\Chowct}[2][]{\operatorname{CH}^{#1}\!\left( #2 \right)}
\newcommand{\Chowcov}[2][]{\operatorname{CH}_{#1}\!\left( #2 \right)}
\newcommand{\Chd}[2][]{\operatorname{Ch}^{#1}\!\left( #2 \right)}

\newcommand{\rest}[2]{{#1}_{\mid #2}}
\newcommand{\accol}[1]{\left\{#1\right\}}
\newcommand{\vabs}[1]{\left| #1\right|}
\newcommand{\si}{\text{if}}
\newcommand{\prtent}[1]{\left\lfloor #1 \right\rfloor}

\title{A quadratically enriched count of lines on a smooth del Pezzo surface of degree 5}
\author{CHACHAY Victor} 
\address{Université Bourgogne Europe, CNRS, IMB UMR 5584, 21000 Dijon, France}
\email{victor.chachay@ube.fr}
\date{July 2026}

\begin{document}

\begin{abstract}
In this paper, we give an enriched count for the 10 lines on a smooth del Pezzo surface $S$ of degree 5 as the quadratic degree of the Euler class of a locally free sheaf. To do this, we see $S$ as a section of the quintic del Pezzo threefold $V_5$ and the lines on $S$ as a sub-scheme of the $\eP^2$ of lines on $V_5.$ We then use some oriented Schubert calculus and Bott residue formula with $N_{\SL_2}-$action to compute the degree of the Euler class.
\end{abstract}

\maketitle

\tableofcontents

\section{Introduction}
 
\addtocontents{toc}{\protect\setcounter{tocdepth}{1}}

The count of lines on a smooth del Pezzo surface, over an algebraically closed field, is a well-known enumerative invariant. For a del Pezzo surface $S$ of degree 5, there are 10 lines, which are the $(-1)-$curves. Over a perfect field $\ek$ of characteristic different from 2, the 10 lines of $S$ may only exist over $\overline{\ek}.$ To recover an enumerative invariant over $\ek,$ we use the quadratically enriched intersection theory, namely the Chow-Witt groups. We do it in the same fashion as the enriched count for the 27 lines on a cubic surface from \cite{Kass_Wickelgren}. We realise the Hilbert scheme of lines on a del Pezzo surface as the zero-locus of a global section of a sheaf $\cE$ over a scheme (here, it will be $\eP^2$). As $\cE$ is relatively orientable, the degree of its Chow-Witt Euler class is an element in the Grothendieck-Witt group of the quadratic forms $\GW(\ek)$ and is our quadratically enriched count. 

\begin{thintro}[Main Theorem \ref{THO - classe euler dP 5}]
Let $S$ be a smooth del Pezzo surface of degree 5 over a perfect field $\ek$ and $\mathrm{char}(\ek) \neq 2$. An enriched count of lines on $S$ is given by
\[\deg^{CW} e(\cE) = 2 + 4h \in \GW(\ek)\]
where $h = 1 + \qscal{-1}$ is the hyperbolic element in $\GW(\ek).$
\end{thintro}

More explicitly, over an algebraically closed field, there is only one del Pezzo surface of degree 5, up to isomorphism, and it can be seen as a linear section of the Grassmannian $\Grass(5,2)$ in its Plücker embedding in $\eP^9.$ It may also be seen as a blow-up of $\eP^2$ in $4$ points in general position. From this blow-up characterisation, one gets the 10 lines arranged in a Petersen diagram. The first characterisation still holds over any field.
\begin{propintro}[Proposition \ref{PROP - plongement dP dans V5 dans P9} in the text]
Let $S$ be a smooth del Pezzo surface of degree $5$ over a field $\ek.$ It embeds into a Fano threefold $V_5$ which in turn embeds in the Grassmannian $\Grass(5,2).$
\end{propintro}
By Fano threefold $V_5,$ we mean a Fano threefold of index 2, Picard rank 1 and degree 5. In particular, there is only one such variety, up to isomorphism, over an algebraically closed field. However, there are different $\ek-$forms possible in general. This remark on $\ek-$forms is what breaks the blow-up description in general: a degree 5 del Pezzo surface $S$ may not be isomorphic to a blow-up of $\eP^2$ and only $S_{\overline{\ek}}$ may be.

From \cite{FN89}, the Hilbert scheme $\fH$ of lines on $V_5$ is isomorphic to $\eP^2$, even over a non-algebraically closed field. We use the incidence correspondence of Diagram (\ref{EQU - diagramme incidence Fano})
\begin{align*}
    \xymatrix{&&\fU\ar[ld]_\psi \ar[rd]^\pi & \\ S \hspace{0.1cm}\ar@{^{(}->}[r] & V_5&&\fH}
\end{align*}
with the del Pezzo surface $S$ in a $V_5$ and $\fU$ the universal family of lines on $V_5.$ With some pullback and push-forward along this incidence correspondence, as explained in \cite{Cha_dP2-4}, one gets the Hilbert scheme of lines on $S$ as the zero-locus of a section of a locally free sheaf $\cE$ of rank $2$ over $\eP^2.$

This enriched count gives back the classical result of 10 lines on an algebraically closed field as the rank of this quadratic form. An analogy with real signed counts and $Pin^- -$structures has been made for the cubic surface in \cite{Kass_Wickelgren}. We will give an heuristic explanation of how a signed count over a real del Pezzo surface could be done, following the ideas of \cite{FK_real-lines}.

\subsection*{Real signed count heuristic}
If $X$ is a smooth del Pezzo surface defined over $\eR,$ it is known that it has at least 2 real lines (see \cite{Boitrel_dP5} for an extensive description of the lines configuration in the Petersen diagram) and is always a blow-up of $\eP^2$ in real or complex conjugate points. The real locus $X(\eR)$ is either homeomorphic to $\sharp^5 \eP^2_\eR,\ \sharp^3\eP^2_\eR$ or $\eP^2_\eR$ if $X$ contains $10,$ 4 or 2 real lines respectively. From \cite{Boitrel_dP5}, we extract all three possibilities for $X$ defined over $\eR$ and the figures (\subref{fig:figure(0)_option_Gal(kbarre/k)-action_on_Pikbarre}), (\subref{fig:figure(a)_option_Gal(kbarre/k)-action_on_Pikbarre}) and (\subref{fig:figure(b)_option_Gal(kbarre/k)-action_on_Pikbarre}). In particular, it gives us the lines that are defined over $\eR$ in red. The others are complex conjugate lines and thus come in pairs of same color. The labels in the Petersen diagrams are $E_i$ for the exceptional divisors over $p_i$ and $D_{ij}$ for the strict transforms of lines passing through the points $p_i$ and $p_j$ (with $p_i$ the blown-up points over $\bar \ek$).

\vspace{-0.2cm}
\definecolor{forestgreen}{rgb}{0,0.65,0}
\definecolor{cyan}{rgb}{0.871,0.494,0}
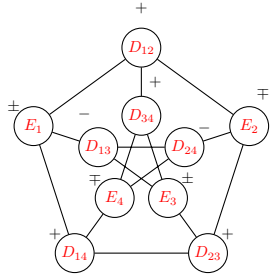
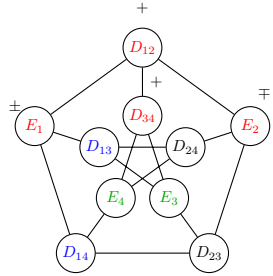
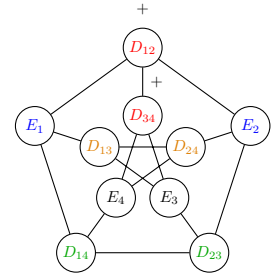
\begin{figure}[h]
  \centering
  \begin{subfigure}[t]{0.24\textwidth}
    \centering
    \tikzmath{
      real \a,\b;
      \a=1.5;
      \b=0.6;
    }
    \begin{tikzpicture}[every node/.style={inner sep=2ex, scale=0.6}, 	rotate=18]
      \foreach \i in {0,1,...,4}{
        \path (\i*72:\a) node[draw, circle] (N-\i) {};
        \path (\i*72:\b) node[draw, circle] (Q-\i) {};
        \draw (N-\i) -- (Q-\i);
      }
      \draw (N-0) -- (N-1) -- (N-2) -- (N-3) -- (N-4) -- (N-0);
      \draw (Q-0) -- (Q-2) -- (Q-4) -- (Q-1) -- (Q-3) -- (Q-0);
      \path (N-0.center) node[red]{$E_2$};
      \path (N-0.north east) node[above]{$\mathbf{\mp}$};
      \path (N-1.center) node[red] {$D_{12}$};
      \path (N-1.north) node[above]{$\mathbf{+}$};
      \path (N-2.center) node[red] {$E_1$};
      \path (N-2.west) node[above]{$\mathbf{\pm}$};
      \path (N-3.center) node[red] {$D_{14}$};
      \path (N-3.west) node[above]{$\mathbf{+}$};
      \path (N-4.center) node[red] {$D_{23}$};
      \path (N-4.east) node[above]{$\mathbf{+}$};
      \path (Q-0.center) node[red] {$D_{24}$};
      \path (Q-0.east) node[above]{$\mathbf{-}$};
      \path (Q-1.center) node[red] {$D_{34}$};
      \path (Q-1.north east) node[above]{$\mathbf{+}$};
      \path (Q-2.center) node[red] {$D_{13}$};
      \path (Q-2.north west) node[above]{$\mathbf{-}$};
      \path (Q-3.center) node[red] {$E_4$};
      \path (Q-3.west) node[above]{$\mathbf{\mp}$};
      \path (Q-4.center) node[red] {$E_3$};
      \path (Q-4.east) node[above]{$\mathbf{\pm}$};
    \end{tikzpicture}    	
    \caption{$\rho(\Galois(\eC/\eR))=\lbrace \Id \rbrace$}
    \label{fig:figure(0)_option_Gal(kbarre/k)-action_on_Pikbarre}
  \end{subfigure}
  \hfill
  \begin{subfigure}[t]{0.24\textwidth}
    \centering
    \tikzmath{
      real \a,\b;
      \a=1.5;
      \b=0.6;
    }
    \begin{tikzpicture}[every node/.style={inner sep=2ex, scale=0.6}, 	rotate=18]
      \foreach \i in {0,1,...,4}{
        \path (\i*72:\a) node[draw, circle] (N-\i) {};
        \path (\i*72:\b) node[draw, circle] (Q-\i) {};
        \draw (N-\i) -- (Q-\i);
      }
      \draw (N-0) -- (N-1) -- (N-2) -- (N-3) -- (N-4) -- (N-0);
      \draw (Q-0) -- (Q-2) -- (Q-4) -- (Q-1) -- (Q-3) -- (Q-0);
      \path (N-0.center) node[red]{$E_2$};
      \path (N-0.north east) node[above]{$\mathbf{\mp}$};
      \path (N-1.center) node[red] {$D_{12}$};
      \path (N-1.north) node[above]{$\mathbf{+}$};
      \path (N-2.center) node[red] {$E_1$};
      \path (N-2.west) node[above]{$\mathbf{\pm}$};
      \path (N-3.center) node[blue] {$D_{14}$};
      \path (N-4.center) node[] {$D_{23}$};
      \path (Q-0.center) node[] {$D_{24}$};
      \path (Q-1.center) node[red] {$D_{34}$};
      \path (Q-1.north east) node[above]{$\mathbf{+}$};
      \path (Q-2.center) node[blue] {$D_{13}$};
      \path (Q-3.center) node[forestgreen] {$E_4$};
      \path (Q-4.center) node[forestgreen] {$E_3$};
    \end{tikzpicture}    	
    \caption{$\rho(\Galois(\eC/\eR)) \simeq \eZ/2\eZ$}
    \label{fig:figure(a)_option_Gal(kbarre/k)-action_on_Pikbarre}
  \end{subfigure}
  \hfill
  \begin{subfigure}[t]{0.24\textwidth}
    \centering
    \tikzmath{
      real \a,\b;
      \a=1.5;
      \b=0.6;
    }
    \begin{tikzpicture}[every node/.style={inner sep=2ex, scale=0.6}, 	rotate=18]
      \foreach \i in {0,1,...,4}{
        \path (\i*72:\a) node[draw, circle] (N-\i) {};
        \path (\i*72:\b) node[draw, circle] (Q-\i) {};
        \draw (N-\i) -- (Q-\i);
      }
      \draw (N-0) -- (N-1) -- (N-2) -- (N-3) -- (N-4) -- (N-0);
      \draw (Q-0) -- (Q-2) -- (Q-4) -- (Q-1) -- (Q-3) -- (Q-0);
      \path (N-0.center) node[blue]{$E_2$};
      \path (N-1.center) node[red] {$D_{12}$};
      \path (N-1.north) node[above]{$\mathbf{+}$};
      \path (N-2.center) node[blue] {$E_1$};
      \path (N-3.center) node[forestgreen] {$D_{14}$};
      \path (N-4.center) node[forestgreen] {$D_{23}$};
      \path (Q-0.center) node[cyan] {$D_{24}$};
      \path (Q-1.center) node[red] {$D_{34}$};
      \path (Q-1.north east) node[above]{$\mathbf{+}$};
      \path (Q-2.center) node[cyan] {$D_{13}$};
      \path (Q-3.center) node[] {$E_4$};
      \path (Q-4.center) node[] {$E_3$};
    \end{tikzpicture}    	
    \caption{$\rho(\Galois(\eC/\eR)) \simeq \eZ/2\eZ$}
    \label{fig:figure(b)_option_Gal(kbarre/k)-action_on_Pikbarre}
  \end{subfigure}  \hfill
  \caption{Configurations of lines and $Pin^-$structure}\label{FIG - ensemble des figures}
 \end{figure}

As highlighted in \cite[Section 2.4]{FK_real-lines}, one can define a $Pin^- -$structure on $X(\eR)$ as follows.
\begin{defintro}[$Pin^- -$structure]\label{DEF - pin structure}
If $X$ is a smooth del Pezzo surface defined over $\eR,$ a $Pin^- -$structure on $X(\eR)$ is an application 
\[\varphi : H_1(X(\eR),\eZ/2\eZ) \to \eZ/4\eZ\]
such that $\varphi(x+y) = \varphi(x) + \varphi(y) +2(x\cdot y)$, with $x\cdot y$ the intersection product.
\end{defintro}
In particular, for $\cml$ a real line of $X,$ if we note $[\cml]$ the class of $\cml(\eR)$ in $H_1(X(\eR),\eZ/2\eZ)$, we get
\[0 = \varphi(2[\cml]) = 2\varphi([\cml]) + 2 [\cml]^2 = 2\varphi([\cml]) + 2 \]
in $\eZ/4\eZ.$ Thus, $\varphi([\cml])$ has to be 1 or 3.
\begin{defintro}[Hyperbolic and elliptic line]\label{DEF - hyperbolic and elliptic line}
A real line $\cml$ in a smooth del Pezzo surface of degree 5 is called
\begin{itemize}
    \item $\varphi-$hyperbolic if $\varphi([\cml]) = 1,$
    \item $\varphi-$elliptic if $\varphi([\cml]) = 3$
\end{itemize}
in $\eZ/4\eZ.$
\end{defintro}
A signed count of the real lines of $X$, with respect to $\varphi,$ is
\[\Sigma_\varphi = \sum_{\cml\  \text{real}} (-1)^{\frac{\varphi([\cml])-1}{2}} = \sharp \accol{\text{real hyperbolic lines}} - \sharp \accol{\text{real elliptic lines}}.\]

If we impose the normalisation $\varphi(w_1(X(\eR)) = 1$ for the first Stiefel-Whitney class and the vanishing of classes of \emph{real roots}\ as in \cite{FK_real-lines} there is only one possible $\varphi$ if $X(\eR)$ contains two lines and two possible $Pin^- -$structures when there are 4 or 10 lines. In every case, we get 
\[\Sigma_\varphi = \sharp \accol{\text{real hyperbolic lines}} - \sharp \accol{\text{real elliptic lines}} = 2.\]

To clarify a bit what happens, let us work out the 4 real lines case. In this case, $H_1(X(\eR),\eZ/2\eZ)$ is spanned by the classes $h,$ coming from the hyperplane section of $\eP^2$ and $e_1$ and $e_2$ coming from the exceptional divisors of the two real points blown-up. Then $w_1(X(\eR)) = h + e_1 +e_2$ and we impose its image to be 1. The 4 real lines are the two exceptional divisors, the strict transform of the line passing through the two real blown-up points ($D_{12}$ in Figure \subref{fig:figure(a)_option_Gal(kbarre/k)-action_on_Pikbarre}) and the strict transform of the line passing through the pair of complex conjugate points $D_{34}$ in Figure \subref{fig:figure(a)_option_Gal(kbarre/k)-action_on_Pikbarre}). Their classes are $h+e_1 +e_2$ and $h$ respectively. The image of $\varphi$ is controlled by the values $\varphi(h),\varphi(e_1),\varphi(e_2)$ in $\accol{1,3}$ and the conditions:
\[\varphi(h+e_i) = \varphi(h)+\varphi(e_i),\ \varphi(e_1+e_2) = \varphi(e_1)+\varphi(e_2).\]

With the normalisation condition, we get $\varphi(h) +\varphi(e_1)+\varphi(e_2) = 1$ and the root condition (c.f. \cite[Section 2.1 and Section 3.2]{FK_real-lines}) imposes $\varphi(e_1+e_2) = 0.$ It leaves us with the choice of $\varphi(e_1)$ determining everything. In Figure \ref{FIG - ensemble des figures}, the $Pin^- -$structure values are shown as the signs on top of the real lines. The sign $+$ corresponds to $+1$ and the hyperbolic lines, the sign $-$ corresponds to $-1 = 3$ in $\eZ/4\eZ$ and the elliptic lines.

For the two possibilities in the configuration with 10 real lines, only the signs of the exceptional divisors vary. It does not change the count and it is compatible with the two possibilities when there are 4 real lines. Indeed, passing from the case in figure (\subref{fig:figure(0)_option_Gal(kbarre/k)-action_on_Pikbarre}) to (\subref{fig:figure(a)_option_Gal(kbarre/k)-action_on_Pikbarre}), the signs for the real lines coincide and the lines that become complex conjugate (pairs of same color) come as pairs of 1 and $-1$. However, it does not coincide with the conjugate pairs in case (\subref{fig:figure(b)_option_Gal(kbarre/k)-action_on_Pikbarre}).

\subsection*{Layout of the paper}

In Section \ref{SEC - motivic setup} we recall the motivic setup form \cite{Cha_dP2-4}. Here we also added a part on $G-$equivariant six functors formalism as it will be used for the quadratic degree computation. In Section \ref{SEC - degre quadratique} we recall identifications that are often implicitly used when computing the quadratic degree of an Euler class in a Chow-Witt group.

We treated the geometry and an embedding of del Pezzo surfaces of degree 5 in the introduction. The main difficulty is then to compute the quadratic degree of the Euler class. We choose to compute this degree via the Atiyah-Bott formula from \cite{LevineAB24}. It needs a $N-$action on the $\eP^2$ of lines on $V_5$ and a $N-$linearisation on $\cE.$ The $N-$action and Atiyah-Bott formula to recover a quadratic degree are explained in Section \ref{SEC - cohom Witt equivariante}. A problem arises when working over a general field: one may not have a $N-$linearisation of $\cE$. Section \ref{SEC - geometrie droites dP5} is dedicated to the geometrical considerations. We need to embed the Hilbert scheme $H$ of lines on $S$ in the $\eP^2$ of lines on $V_5.$ We also explain the problem that may arise to define the action and an embedding of these Hilbert schemes in Grassmannians. This last embedding leads to Section \ref{SEC - calcul Schubert oriente} where we realise $H$ as an oriented Schubert calculus computation. It does not give the quadratic degree but an independence of the computation with respect to the $\ek-$form of $H$ and $S$ at hand. It gives enough justification to compute the Bott residue formula in Section \ref{SEC - calcul Witt schema Hilbert} and state that the result is valid in any case.

\subsection*{Conventions}

We work over a perfect field $\ek$ of characteristic different from 2. The base schemes will be quasi-compact and quasi-separated and we call \emph{s-morphism} a separated morphism of finite type. The convention for vector bundles and projective bundles is the contravariant one (i.e. $\eV(\cE) := \Spec{\Symalg{\cE}}$ and $\eP(\cE) := \Proj{\Symalg{\cE}}$ for $\cE$ a locally free sheaf over a scheme $X$). We will also use curvy letters to refer to locally free sheaves and straight ones for the associated vector bundle.

\subsection*{Acknowledgments}
The author is very thankful to his PhD advisors Adrien Dubouloz and Jan Nagel for the fruitful discussions and comments while writing this paper as well as the opportunity to accomplish it as part of his PhD.

\addtocontents{toc}{\protect\setcounter{tocdepth}{2}}
\section{Motivic setup}\label{SEC - motivic setup}

This first reminder part is mostly taken from \cite{DJK21} where the notations were adapted, see also \cite{Cha_dP2-4}.

\subsection{Stable homotopy and six functors}\label{SSEC - homotopie stable et six foncteurs}

In the following, the base schemes will be quasi-compact and quasi-separated and we call \emph{s-morphism} a separated morphism of finite type.

Let $S$ be a quasi-compact and quasi-separated noetherian scheme of finite dimension. Then we write $\StableHo{S}$ the stable $\infty-$category of motivic spectra.

\begin{prop}
Let $S$ be a quasi-compact and quasi-separated noetherian scheme of finite dimension. Then the category $\StableHo{S}$ is equivalent, in terms of triangulated category, to the $\Au-$homotopy category originally built by Voevodsky.
\end{prop}

The unit for $\otimes$ in $\StableHo{S}$ is noted $\eS_S$ and is the stabilisation of $\eUn_S$.

\begin{prop}[\cite{Ayoub_6foncteurs}]
For every morphism $f : T \to S$, we get a pair of adjoint functors:
\[\xymatrix{f^\ast : \StableHo{S} \ar@<+2pt>[r] & \StableHo{T} : f_\ast, \ar@<+2pt>[l]}\]
called inverse image (by left Kan extension) and direct image. Moreover, if $f$ is a s-morphism, then we get another adjoint pair:
\[\xymatrix{f_! : \StableHo{T} \ar@<+2pt>[r] & \StableHo{S} : f^!, \ar@<+2pt>[l]}\]
called direct (and inverse) exceptional image.
\end{prop}

\begin{tho}[Six functors on $\StableHo{S}$, \cite{Ayoub_6foncteurs}]\label{THO - formalisme six foncteurs SH}
The six operations $(\otimes,\underline{\operatorname{Hom}},f^\ast,f_\ast,f_!,f^!)$ satisfy to the six functors formalism:
\begin{enumerate}
\item For every morphism $f$, $f^\ast$ are symmetric monoidal.
\item There is a natural transformation $f_! \to f_\ast$ invertible when $f$ is proper.
\item There is an invertible natural transformation $f^\ast \to f^!$ when $f$ is an open immersion.
\item There is a canonical isomorphism
\[\eE \otimes f_!(\eF) \to f_!(f^\ast(\eE)\otimes \eF)\]
for all s-morphisms $f : T \to S$ and $\eE$ in $\StableHo{S}$, $\eF$ in $\StableHo{T}$.
\item For every Cartesian square
\[\xymatrix{T' \ar[r]^g \ar[d]_q & S' \ar[d]^p \\ T\ar[r]_f & S,}\]
if $f$ and $g$ are s-morphisms, then we get canonical isomorphisms:
\[p^\ast f_! \to g_!q^\ast, \qquad q_\ast g^! \to f^!p_\ast.\]
\end{enumerate}
\end{tho}

\begin{deff}[Suspension by a locally free sheaf]\label{DEF - suspension par faisceau localement libre}
To a locally free sheaf $\cE$ of finite rank over $S$, we can associate an auto-equivalence 
\[\Sigma^\cE : \StableHo{S} \to \StableHo{S}\]
called the suspension by $\cE.$
\end{deff}

\begin{prop}[Thom space of a locally free sheaf]\label{PROP - def foncteur Thom}
Let us define $\operatorname{Th}$ the functor which, to a locally free sheaf $\cE$ of finite rank over $S$, associates $\Thom[S]{\cE} := \Sigma^\cE(\eS_S)\in \Picard{\StableHo{S}}.$ It can be naturally extended as a morphism of $\cE_\infty$-groups:
\[\operatorname{Th}: K(-) \to \Picard{\StableHo{-}}.\]
\end{prop}

\begin{rmq}
Rather than considering vector bundles, when possible we chose to keep the suspension by Thom spaces of locally free sheaves. Thus, we replace the $K-$theory of vector bundles by the one of locally free sheaves (Grothendieck $K_0$ group).
\end{rmq}

\begin{rmq}
Let $\cE'\to \cE \to \cE''$ be a short exact sequence of locally free sheaves. It canonically gives rise to an isomorphism $\qscal{\cE} \simeq \qscal{\cE'} + \qscal{\cE''}$ in $K_0(S)$. We then get $\Thom[S]{\cE} \simeq \Thom[S]{\cE'}\otimes\Thom[S]{\cE''}$ in $\StableHo{S}$.
\end{rmq}

\begin{tho}[Homotopic purity]\label{THO - purete version six foncteurs}
If $f$ is a smooth s-morphism, there exists a canonical isomorphism of functors 
\[\fp_f : \Sigma^{\cT_f}f^\ast \to f^!\]
with $\cT_f$ the relative tangent sheaf of $f$.
\end{tho}

\begin{rmq}
If a $s-$morphism $f$ is a smoothable local complete intersection (i.e. admits a global factorisation into a regular closed immersion followed by a smooth morphism) between s-schemes, then there is a purity isomorphism
\[\fp_f : \Sigma^{\cL_f}f^\ast \to f^!\]
with $\cL_f$ the virtual tangent sheaf associated to the conormal complex of $f$ (c.f. \cite[Paragraph 4.3.1]{DJK21}).
\end{rmq}

\subsection{Universal Euler class}\label{SSEC - classe Euler universelle}

Let us recall the characterisation of the Euler class that was given in \cite[Section 2]{Cha_dP2-4} using only the tools from six functor formalism.\\

Let $X \to S$ be a smooth s-morphism and a locally free sheaf $\cE$ of constant rank $r$ on $X$. We then get schemes morphisms $p : E = \eV(\cE) \to X$ and $s : X \to E$, the projection from the vector bundle and the zero section, respectively. The virtual tangent sheaf to the zero section $\cL_s$ is actually
the conormal sheaf $(\cE/\cE^{2})$.
Moreover, since $\Symalg{\cE} \simeq \Symalg{\cL_s}$ it gives rise to the purity isomorphism \ref{THO - purete version six foncteurs}: $\Sigma^{\cE} s^\ast \isomr s^!$.

\begin{deff}[Stable universal Euler class]\label{DEF - classe Euler univ stable}
Given the above notations, the counit $\epsilon$ of the adjonction $\xymatrix{s^\ast : \StableHo{E} \ar@<+2pt>[r] & \StableHo{X} : s_{\ast}, \ar@<+2pt>[l]}$, the unit $ \eta$ of the adjonction $\xymatrix{s_! : \StableHo{X} \ar@<+2pt>[r] & \StableHo{E} : s^{!}, \ar@<+2pt>[l]}$ and the purity theorem \ref{THO - purete version six foncteurs} together give the stable universal Euler class as the composition:
\[\xymatrix{e^u(\cE) : \eS_X \ar[r]^\eta & s^! s_{!}\eS_X \ar[r]^{\fp_{s}\qquad \quad } & \Thom[X]{\cE}\otimes_X s^\ast s_{\ast}\eS_X \ar[r]^{\qquad\epsilon}&\Thom[X]{\cE} .}\]
\end{deff}

\subsection{Homologies, cohomologies}

In the context of the stable homotopy and six functors formalism, we define homologies / cohomologies as morphism classes:

\begin{deff}[Homologies and cohomologies]\label{DEF - homologies et cohomologies}
Let us fix a base scheme $S$ and $\eE$ a motivic spectrum in $\StableHo{S}$. For $f : X \to S$ a morphism and $v$ in $K(X)$ (c.f. Proposition \ref{PROP - def foncteur Thom}), we define:
\begin{align*}
\text{(Borel-Moore homology)}\ &\ \eE^{BM}(X/S,v) := \ulHom[\StableHo{S}]{\eS_S,f_\ast(f^!(\eE) \otimes\Thom[X]{v})}\\
\text{(cohomology)}\ &\ \eE(X,v) := \ulHom[\StableHo{S}]{\eS_S,f_\ast(f^\ast(\eE) \otimes\Thom[X]{v})}\\
\end{align*}
\end{deff}

\begin{deff}[Euler class with coefficients]\label{DEF - Euler class with coefficients}
Let $f: X \to S$ be a smooth s-morphism and a locally free sheaf $\cE$ of rank $r$ on $X$. Then, for $\eE$ a unital theory on $S$, the Euler class \ref{DEF - classe Euler univ stable} $e^u(\cE)$ defines an element $e^\eE(\cE)$ of $\eE(X,\cE)\simeq \eE^{BM}(X/S,\cE-\cT_f).$ We call this element the Euler class with coefficients (in $\eE$) and write $e(\cE)$ if there is no ambiguity on $\eE.$
\end{deff}

\begin{prop}[Proper push-forward]\label{PROP - covariance propre}
Let $\eE$ be a theory in $\StableHo{S}.$ For every proper morphism $f : X \to Y$ of s-schemes on $S,$ there is a direct image 
\[f_\ast : \eE^{BM}(X/S,f^\ast v) \to \eE^{BM}(Y/S,v).\]
\end{prop}

\begin{proof}
See \cite[Proposition 2.3.4]{Cha_dP2-4}.
\end{proof}
\subsection{Orientations}\label{SSEC - orientations}
Let us recall here orientations on the spectra we need. We also stress out the difference with the relative orientability of a locally free sheaf.

\begin{prop}[Orientation of usual theories]
Let $S = \Spec{\ek}$ be the base scheme.
\begin{itemize}
    \item The Milnor-Witt cohomology spectrum is $\SL^c-$oriented in $\StableHo{\ek}$ (c.f. \cite[Remark 5.5]{Ananyevskiy_SLoriented}).
    \item The Witt cohomology spectrum (or of the $\fId^\bullet-$cohomology) is $\SL^c-$oriented in $\StableHo{\ek}$.
    \item The Milnor cohomology spectrum is $\Gl-$oriented in $\StableHo{\ek}$.
\end{itemize}
\end{prop}

\begin{example}
Let $\cE$ be a locally free sheaf of rank $r = \dim X$ over a smooth and s-scheme $X$ over $\ek$.
\begin{itemize}
    \item In the Chow groups (i.e. cohomology with coefficients in the Milnor spectrum), the Euler class $e(\cE)$ becomes $c_r(\cE)$ in $\Chowct[\dim X]{X} \simeq \Chowcov[0]{X}.$
    \item In the Chow-Witt groups (i.e. cohomology with coefficients in the Milnor-Witt spectrum), the Euler class $e(\cE)$ lands in $\Chowtilde[r]{X,\det \cE}$ which is isomorphic to $\Chowtildecov[0]{X,\det \cE \otimes \omega_X^\vee},$ with $\omega_X := \det \Omega_X$ the determinant of the sheaf of Kähler differentials on $X.$
\end{itemize}
\end{example}
The last twist in the Chow-Witt group gives an obstruction to use the proper push-forward along the structural map $X \to \Spec{\ek}.$

\begin{deff}[Relative orientability]\label{DEF - orientabilite relative}
Let $X$ be a smooth $\ek-$scheme and $\cE$ a locally free sheaf of finite rank $r = \dim X$ on it. We say that $\cE$ is relatively orientable if there exists $\cL$ a locally free sheaf of rank 1 such that $\det \cE \simeq \omega_X \otimes \cL^{\otimes 2}.$ 

We also say that two invertible sheaves are in the same quadratic equivalence class or quadratically equivalent if they differ only by the product with a tensor square and we will note $\cE \simquad \cF$. In particular, we say that $\cE$ is quadratically trivial if quadratically equivalent to $\cO_X.$
\end{deff}

\begin{deff}[Quadratic degree]\label{DEF - degre quadratique}
Let $\pi : X\to \Spec{\ek}$ be a smooth and proper $\ek-$scheme. The quadratic degree is the push forward (c.f. Proposition \ref{PROP - covariance propre})
\[\pi_\ast : \Chowtildecov[0]{X} \to \Chowtildecov[0]{\ek}\simeq \GW(\ek).\]
\end{deff}

\begin{rmq}\label{RMQ - degre quad classe Euler}
The Chow-Witt Euler class $e(\cE)$ of a locally free sheaf $\cE$ of rank $n = \dim X$ over a smooth scheme $X$ lives in $\Chowtilde[n]{X,\det \cE}.$  To talk about the quadratic degree of an Euler class, we need to see this class in $\Chowtildecov[0]{X}.$ If $\cE$ is relatively orientable, a choice of orientation is an isomorphism $\rho : \det \cE \isomr \omega_X\otimes \cL^{\otimes 2}$ inducing an isomorphism 
\[\Chowtilde[n]{X,\det \cE} \simeq_\rho \Chowtildecov[0]{X}.\]
Only then can we take the quadratic degree and interpret the push forward of the Euler class as an element in the Grothendieck-Witt group.

We will also talk about relative orientation of $\Chowtilde[\dim X]{X,\cL}$ if $\cL \simquad \omega_X$ as we sometimes focus on the cohomology group directly. 
\end{rmq}

\begin{prop}[Orientability over the Grassmannian]\label{PROP - orientabilite grassmann}
Let $\Grass(n,k)$ be the Grassmannian of the rank $k$ quotients of $V$ a $n-$dimensional vector space over $\ek$. Then, $\omega_G \simeq (\det \cQ^\vee)^{\otimes n} := \cO(-n)$ in $\Picard{G}$. 
\end{prop}

\begin{proof}See \cite[Proposition 2.4.7]{Cha_dP2-4}.

\end{proof}

\begin{rmq}
From the last proposition, we emphasise on the fact that up to quadratic equivalence, there are only two possible orientations on the Grassmannian, the trivial one $\cO$ and the non-trivial one $\det \cQ.$
\end{rmq}

\subsection{Six functors with $G-$action}
We will now adapt all of this setup into the $G-$equivariant stable homotopy, using \cite{Hoyois_G-equiv} results.

Let $B$ be a base scheme and $G$ a flat group scheme of finite presentation over $B$.

\begin{deff}[$G-$solvable scheme]\label{DEF - schema G resoluble}
Let $X$ be a $G-$scheme. we say that $X$ has the property of $G-$resolution if for all $G-$module quasi-coherent of finite type $\cM$ over $X$, there exists a locally free $G-$module $\cE$ of finite rank and an epimorphism $\cE \twoheadrightarrow\cM$.
\end{deff}

\begin{deff}[Linearly reductive scheme]\label{DEF - schema lin reductif}
A group scheme $G$ over $B$ is called linearly reductive if the functor of $G-$fixed points $(-)^G : \QCoh[G]{B} \to \QCoh{B}$ is exact.
\end{deff}

\begin{deff}[Tamed group schemes]\label{DEF - schema en grp modere}
A flat group scheme $G$ over $B$ is tamed if
\begin{itemize}
\item $B$ admits a Nisnevich cover by $G-$solvable schemes (Definition \ref{DEF - schema G resoluble}),
\item $G$ est linearly reductive (Definition \ref{DEF - schema lin reductif}).
\end{itemize}
\end{deff}

With this definition, we can state the six functors formalism for the $G-$equivariant stable homotopy:

\begin{tho}[Theorem 6.18 in \cite{Hoyois_G-equiv}]\label{THO - six foncteurs equivariants}
Let $B$ be a qcqs scheme and $G$ a tamed group scheme over $B$ (if $G$ is not finite, $B$ has to be $G-$solvable). For $f : T \to S $ a morphism in $\Schcat_B^G$, we have the adjoint pair
\[\xymatrix{f^\ast : \StableHo[G]{S} \ar@<+2pt>[r] & \StableHo[G]{T} : f_\ast. \ar@<+2pt>[l]}\]
Furthermore, if $f$ is separated of finite type (s-morphism), we have :
\[\xymatrix{f_! : \StableHo[G]{T} \ar@<+2pt>[r] & \StableHo[G]{S} : f^!. \ar@<+2pt>[l]}\]
Then, the operators $(\otimes, \underline{\operatorname{Hom}}, f^\ast,f_\ast,f_!,f^!)$ satisfy to the six functors formalism and verify the same property as in \ref{SSEC - homotopie stable et six foncteurs}.
\end{tho}

We can also mimic the construction of the universal Euler class from Section \ref{SSEC - classe Euler universelle} and we get:

\begin{deff}[$G-$equivariant Euler class]\label{DEF - classe Euler equivariante}
With the same notations as in \ref{SSEC - classe Euler universelle} but with all morphisms being $G-$morphisms, we get the $G-$equivariant universal Euler class as the composition
\[\xymatrix{e_G(\cE) : \eS_X \ar[r]^\eta & s^! s_{!}\eS_X \ar[r]^{\fp_{s}\qquad \quad } & \Thom[X]{\cE}\otimes_X s^\ast s_{\ast}\eS_X \ar[r]^{\qquad\epsilon}&\Thom[X]{\cE} .}\]
\end{deff}

\begin{rmq}
With the same definitions of homologies and cohomologies as in \ref{DEF - homologies et cohomologies} but for $G-$equivariant spectra, we get $e_G(\cE)$ as an element of $\eE(X,\cE)$.

Moreover, the orientation considerations as discussed in Section \ref{SSEC - orientations} still hold in the equivariant setup.
\end{rmq}

\section{Quadratic degree}\label{SEC - degre quadratique}
The definitions of Milnor-Witt, Milnor and Witt $K-$theory rings, their relations and cohomologies can be found in \cite[Section 3]{Cha_dP2-4}. They lead to some properties of Chow-Witt groups in maximal codimension and simplifications for the computation of Euler classes. We will just recall some results.

\subsection{About the $K-$theories}

Under the isomorphism (c.f. \cite[Lemma 3.10]{Morel_A1-alg-top}) between $K^{MW}_0(\ek)$ and $\GW(\ek),$ we have an \emph{hyperbolic element}\ in $\GW(\ek)$ defined as the quadratic form class
\[h := 1 + \qscal{-1},\]
with $\qscal{-1}$ the quadratic form class associated to $X \mapsto -X^2.$ In the Milnor-Witt $K-$theory ring notations, $h$ would be written as $2 + \eta [-1],$ with $\eta$ the Hopf element and $[-1]$ the generator of degree 1 in the Milnor-Witt $K-$theory ring, but we will keep the quadratic form version.

\begin{prop}\label{PROP - foncteur H bien def indep torsion}
For $\cL$ a 1-dimensional $\ek-$vector space, the twisted hyperbolic application
\[H:\left| \begin{array}{ccc}
     K^M_\ast(\ek)& \to& \KMW{\ek,\cL} \\
    \sigma & \mapsto & (h\sigma)\otimes l 
\end{array}\right. \]
is well defined and does not depend on the choice of $l$ in $\cL^\times$.
\end{prop}

\begin{proof}
See \cite[Proposition 3.2.4]{Cha_dP2-4}.
\end{proof}

\begin{deff}[Forgetful and hyperbolic applications]\label{DEF - applications oubli et hyperbolique}
The two $\eZ-$graded algebras homomorphisms $F: \KMW{\ek,\cL} \to K^M_\ast(\ek)$ and $H: K^M_\ast(\ek) \to \KMW{\ek,\cL}$ 
are called the forgetful and hyperbolic homomorphisms. They are functorial with respect to the field.
\end{deff}

\begin{rmq}
By definition, $F$ is characterised by $F(\eta) = 0$ and $F([a]\otimes l) = \accol{a}$. $H$ is given by the multiplication by $h$ thus $H(\accol{a}) = h[a]$. Moreover, the compositions give $F\circ H = 2 \Id$ and $H \circ F = \gamma_h$ (the multiplication by $h$). 
\end{rmq}

\begin{prop}\label{PROP - fonction hyperbolique bien def torsion}
For $n$ a non-negative integer, $X$ a smooth $\ek-$scheme and $\cL$ a locally free sheaf of rank 1 over $X,$ the hyperbolic application $H: \Chowct[n]{X} \to \Chowtilde[n]{X,\cL}$ of multiplication by $h$ is well defined and makes the following diagram commute:
\[\xymatrix{\Chowct[n]{X}\ar[d]_{H} \ar@{=}[r] & \Chowct[n]{X} \ar[d]^{\times 2}\\
\Chowtilde[n]{X,\cL} \ar[r]^F & \Chowct[n]{X}. }\]
\end{prop}

\subsection{Chow-Witt Euler classes computation}\label{SSEC - calcul classes Euler}

\begin{prop}\label{PROP - CW comme image H et Witt}
Let $X$ be isomorphic to a Grassmannian or a projective bundle over it. When $n = \dim X,$ a cocycle in $\Chowtilde[n]{X,\cL}$ is the data of its image in $H^n(X,\underline{K}^W_n(\cL))$ and of an element  entirely determined by the value of the induced cocycle in $\Chowct[n]{X}.$ More precisely, we have
\[\Chowtilde[n]{X,\cL} \simeq H^n(X,\sKW_n(\cL))\times_{\Chd[n]{X}}\Chowct[n]{X}.\]
\end{prop}

\begin{proof}
See \cite[Proposition 3.4.1]{Cha_dP2-4}.
\end{proof}


\begin{cor}\label{COR - groupe Chow-Witt corps}
If $X = \Spec{\ek},$ we have obvious isomorphisms (using the previous notations and definitions introduced in this section \ref{SEC - degre quadratique}):
\begin{itemize}
\item $H^0(X,\sKW_0) \simeq \bfW^0(X) \simeq \Witt(\ek),$
\item $\Chowct[0]{X} \simeq \eZ,$
\item $\Chowtilde[0]{X} \simeq \GW(\ek)$.
\end{itemize}
In particular, the Grothendieck-Witt degree can be determined through the decomposition into the Witt degree and Chow degree. The difference is then a multiple of $h$.
\end{cor}

\begin{prop}[Orientable Euler degree]\label{PROP - classe Euler orientable}
Let $\cE$ be an orientable locally free sheaf of rank $r$ over a smooth $\ek-$scheme $X$ of dimension $r$. Then 
\[\deg^{CW} e^{CW}(\cE) = w + n h \in \GW(\ek)\]
where $w$ is the image of $\deg^W e^W(\cE)$ in the Witt group and $n$ is an integer such that $rk(w) + 2n = \deg c_r(\cE).$

\end{prop}

\begin{proof}
If we note $f : X \to \Spec{\ek}$ the structural morphism, then we look at the value of $f_\ast e^{CW}(\cE)$ in $\Chowtildecov[0]{\Spec\ek} = \Chowtilde[0]{\Spec\ek}.$ Then, the value is determined by the value in $H^0(\Spec{\ek},\sKW_0) \simeq \Witt(\Spec{\ek})$ and in $\Chowct[0]{\Spec\ek}$. By the fibre product property of Proposition \ref{PROP - CW comme image H et Witt}, we have the complete description. 
\end{proof}

\begin{rmq}\label{RMQ - donnee Witt est la seule importante}
In the end, when computing Euler classes (or degree), the "quadratic information" is solely contained in the cohomology with coefficients in $K^W$ (or the Witt group when orientable). The other part is then just a translation of the "classical" result over an algebraically closed field back in the Chow-Witt group (or Grothendieck-Witt group).
\end{rmq}

\section{$\bfW-$cohomology as $N-$equivariant cohomology}\label{SEC - cohom Witt equivariante}

\subsection{The choice of the group $N$}
We will specialise everything to the equivariant cohomology with respect to a group $N$ such that we get back the cohomology with coefficients in the Witt group. It will give some means to compute the Euler degree when the Euler class is orientable.

Let's take $N = \left\langle t = \left(\begin{array}{cc}
    t & 0 \\
    0 & t^{-1}
\end{array}\right) , \sigma = \left(\begin{array}{cc}
    0 & 1 \\
    -1 & 0
\end{array}\right) \right\rangle$ the normalisation of the torus in $\SL_2$.

\begin{prop}[Remark 2.4 in \cite{LevineAB24}]
The group $N$ gives a tamed group scheme on $\Spec{\ek}.$
\end{prop}

This way, the ($N-$equivariant) Euler class and the (co)homologies are properly defined the same way as before and every consideration of orientability and quadratic degree still hold. In particular, the degree map lands in the cohomology of the classifying space $BN,$ which we now (partially) describe.

\begin{deff}[$N-$representations]\label{DEF - N rpz}
\cite[Section 6]{Levine_witt-val}
\begin{enumerate}
    \item For $m \geq 1$ an integer,  we define the representation of $N$ on $L(m) = \eA^2$ by: 
    \begin{itemize}
        \item $\rho_m(t) = t^m$ and
        \item $\rho_m(\sigma) = \left(\begin{array}{cc}
    0 & 1 \\
    (-1)^m & 0
\end{array}\right).$
    \end{itemize}
    \item For $m \geq 1$ an integer,  the representation $\rho_m^- : N \to \SL_2$ is given by: 
    \begin{itemize}
        \item $\rho_m^-(t) = \rho_m(t)$ and
        \item $\rho_m^-(\sigma) = (-1)\rho_m(\sigma).$
    \end{itemize}
    \item For $m = 0,$ the representation $L(0)$ is the trivial $1-$dimensional representation of $N$.
\end{enumerate}
\end{deff}

\begin{deff}[Bundle associated to an $N-$representation]\label{DEF - fibre assoc N rpz}
We define $\Tilde{\cO}^\pm(m)$ the rank two vector bundle on $BN$ associated to $\rho_m^\pm.$
\end{deff}
\begin{rmq}\label{RMQ - ordre base rpz impair}
To determine if the representation is the positive one or the negative one when $m$ is odd, we need to choose an order of the basis. The order is always the decreasing one for the weights induced by the action by $t.$
\end{rmq}

\noindent\textbf{Notations:}\ Let $e$ be the Euler class of the tautological rank 2 bundle of $B\SL_2$ in $H^2(B\SL_2,\cW)$. If we note $p : N \hookrightarrow \SL_2$ the immersion of $N$ in $\SL_2$, we then get $p^\ast e$ its pull-back class in $H^2(BN,\cW).$ 

We will also note $\gamma$ a generator of $\Picard{BN}$ as defined in \cite[section 5]{Levine_witt-val}.

\begin{prop}[$BN$ cohomology]\label{PROP - cohomologie BN}
\cite[Proposition 5.5]{Levine_witt-val}

Let $\Witt(F)[x_0,x_2]$ be the polynomial algebra spanned by the $x_i$ in degree $i$. We have an isomorphism
\[H^\ast(BN,\cW) \simeq \Witt(F)[x_0,x_2]/(x_0^2-1, (1+x_0)x_2)\]
where $x_2$ is the image of $p^\ast e.$

\end{prop}

\begin{rmq}
There is also a description for $H^\ast(BN,\cW(\gamma))$ with $\gamma$ a generator of the Picard group $\Picard{BN}$ up to quadratic equivalence. It is described explicitly in \cite[Theorem 5.1]{LevineAB24} but we don't need it for the computations here.
\end{rmq}

\begin{tho}[Theorem 7.1, \cite{Levine_witt-val}]\label{THO - classe euler N rpz}
Suppose $\mathrm{char}(\ek)$ is prime to $m$ if $m \neq 0.$ We define $\epsilon(m) =(-1)^{\prtent{\frac{m}{2}}}$. Then,
\[e^\cW(\Tilde{\cO^+}(m)) = \left\{\begin{array}{cc}
    \epsilon(m)\cdot m\cdot p^\ast e \in H^2(BN,\cW) & \si\ m\ \text{odd} \\
    -\epsilon(m)\frac{m}{2}\cdot\Tilde{e} \in H^2(BN,\cW(\gamma))& \si\ m\ \text{even}.
\end{array}\right.\]
Moreover, $\Tilde{e}^2 = 4p^\ast e^2$ (then an element of $H^4(BN,\cW)$) and $e^\cW(\Tilde{\cO}^-(m)) = -e^\cW(\Tilde{\cO^+}(m)).$
\end{tho}

\subsection{Atiyah-Bott formula}
We will here shorten a bit the statements of \cite{LevineAB24} to give them in the setup we need. In particular, we will only look at smooth schemes and the action of $N.$ So, by \cite[Remark 9.6]{LevineAB24}, most of the hypothesise are automatically checked.

Instead of considering the sub-scheme $X^N$ of $N-$invariant points in a scheme $X$ with $N-$action, we will look at

\begin{deff}[Weakly $N-$invariant points]\label{DEF - points invariants pour N}
Let $T_1$ be the torus of $SL_2$. Let $\vabs{X}^N$ be the set of irreducible components $Z$ of $X^{T_1}$ such that $\Bar{\sigma}Z = Z$ where $\Bar{\sigma}$ is the class of $\sigma$ in $N/T_1$.
\end{deff}

\begin{prop}[Theorem 8.7 in \cite{LevineAB24}]\label{PROP - atiyah-bott loc}
Let $X$ be a $\ek-$scheme with a semi-strict action of $N$. We note $i : \vabs{X}^N \hookrightarrow X$ the inclusion. Let $\cL$ be a $N-$linearised invertible sheaf on $X$. Then, there exists a positive integer $M$ such that 
\[i_\ast : H^{B.M.}_{N,\ast}(\vabs{X}^N,\cW(i^\ast \cL))[(Mp^\ast e)^{-1}] \isomr H^{B.M.}_{N,\ast}(X,\cW(\cL))[(Mp^\ast e)^{-1}].\]
\end{prop}

\begin{rmq}
Inverting $Me$ in every homology group has to be understood as inverting it at the level of coefficients in the base ring. We see the homology groups as modules over $H^{BM}_\ast(BN,\cW) \simeq H^{BM}_{N,\ast}(\Spec{\ek},\cW).$
\end{rmq}

To give the theorem localisation theorem and an idea of how it works, we give two key lemmas \cite[Lemma 9.1 and 9.3]{LevineAB24}:

\begin{lemme}\label{LEM - triangle commutatif produit classe Euler}
Let $\cL$ be an invertible sheaf on $X$ and $i : Y \to X$ a regular $N-$embedding of codimension $r$. The composition
\[i^!i_\ast : H^{B.M.}_{N,\ast}(Y,\cW(i^\ast \cL)) \to H^{B.M.}_{N,\ast-r}(Y,\cW(i^\ast \cL\otimes \det \cN_i))\]
is in fact the intersection product with $e(\cN_i) \in H^r _N(Y,\cW(\det \cN_i))$ with $\cN_i$ the conormal sheaf to $i.$
\end{lemme}

\begin{lemme}\label{LEM - classe generique et inversibilite}
Let $\cV$ be a locally free $N-$sheaf of even rank on $Y \in \mathbf{Sm}^N_\ek$ a connected scheme. There exists a generic class $e^{gen}(\cV)$ in $H^{\ast}(BN,\cW)$ and $e_N(\cV)$ is invertible in $H^\ast_N(Y,\cW(\det \cN_i))[e^{gen}(\cV)^{-1}].$
\end{lemme}

\begin{tho}[Theorem 9.5 \cite{LevineAB24}]\label{THO - residu Bott}
Let $X$ be a smooth scheme in $\mathbf{Sm}^N_\ek$ and $\cL$ be an invertible $N-$linearised sheaf on $X$. We write $i_j : \vabs{X}_j^N \to X$ for the regular embedding of connected components of $\vabs{X}^N$ and $\cN_{i_j}$ conormal sheaf associated to $i_j$. We write $e^{gen}$ for the product of the $e_N^{gen}(N_{i_j})$. Lastly, let $M$ be a positive integer that is a multiple of the exponential characteristic of the base field $\ek$. If $char(\ek)=0$ or $\dim_\ek(\vabs{X}^N) = 0$, then:
\[H^{B.M.}_{N,\ast}(\vabs{X}^N,\cW(i^\ast \cL))[(Mp^\ast e)^{-1}(e^{gen})^{-1}] \cong \prod_j H^{B.M.}_{N,\ast}(\vabs{X}^N_j,\cW(i^\ast \cL))[(Mp^\ast e)^{-1}(e^{gen}(N_{i_j}))^{-1}]\]
and the inverse Atiyah-Bott localisation
\[i_\ast : H^{B.M.}_{N,\ast}(\vabs{X}^N,\cW(i^\ast \cL))[(Mp^\ast e)^{-1}( e^{gen})^{-1}] \isomr H^{B.M.}_{N,\ast}(X,\cW(\cL))[(Mp^\ast e)^{-1}(e^{gen})^{-1}]\]
is given by the Bott residue formula:
\[x \mapsto \prod_j i_j^!(x) \cap e_N(\cN_{i_j})^{-1}.\]
\end{tho}

\begin{rmq}
The last theorem is vastly truncated compared to the original statement since we look only at schemes with action of $N$ and all the intermediate hypotheses are checked. All details and description of the inverted elements can be found in \cite[Section 9]{LevineAB24}.
\end{rmq}

\begin{prop}\label{PROP - residu de Bott specialise points fixes}
If $\vabs{X}^N$ is composed of fixed points, the Bott residue formula in Theorem \ref{THO - residu Bott}, when applied to the Euler class $e_N^\cW(\cE)$ for a locally free sheaf $\cE$ over a scheme $X$ gives:
\[ e_N^\cW(\cE) = \sum_{p_j\ \text{fixed points}} e^\cW_N(\rest{\cE}{p_j})\cdot e^\cW_N(\rest{\cT_X}{p_j})^{-1}.\]
\end{prop}
\begin{proof}
The equality is already proven by the Bott residue formula in Theorem \ref{THO - residu Bott} except for the identification $i_p^!(e(\cE)) \simeq e(\rest{\cE}{p})$ with $i_p : p \to X$ the inclusion of a fixed point.

Let us consider the Cartesian square
\[\xymatrix{\rest{E}{p} \ar[r]^{\iota} & E\\ 
\accol{p} \ar[u]^{\sigma} \ar[r]^{i} & X \ar[u]_s}\]
where $E = \eV(\cE)$ is the vector bundle associated to $\cE.$ From the definition of the Euler class (c.f. Definition \ref{DEF - classe Euler equivariante}), it suffices to show that $i^!s^!s_! \simeq \sigma^!\sigma_!i^\ast.$

It follows from
\begin{align*}
    i^!s^!s_!(\eS_X)& = (s i)^!s_!\eS_X \\
    & \simeq \sigma^!\iota^! s_\ast \eS_X \text{by commutativity and}\ s\ \text{proper}\\
    & \simeq \sigma^!\sigma_!i^\ast \eS_X\ \text{by base change}\\
    & \simeq \sigma^! \sigma_!(\eS_{\accol{p}}).
\end{align*}
\end{proof}

\begin{cor}\label{COR - simplification AB}
From Theorem \ref{THO - residu Bott}, assume further that $\vabs{X}^N = X^N$ and that $\cL$ is quadratically trivial.
Then the Witt degree of a $0-$class $c$ in $H^{BM}_{N,0}(X,\cW(\cL))[(Mp^\ast e)^{-1}(e^{gen})^{-1}]$ is the sum of the degrees of $i_j^!(c) \cap e_N(N_{i_j})^{-1}$ in $H^0(BN,\cW)[(Mp^\ast e)^{-1}(e^{gen})^{-1}]$.
\end{cor}

\section{Lines in a degree 5 del Pezzo surface}\label{SEC - geometrie droites dP5}

\subsection{General facts about degree 5 del Pezzo surfaces}

We will first recall some details about characterisations of smooth del Pezzo surfaces of degree 5 over a perfect field $\ek$ of characteristic different from 2 . The main point is to embed $S$ into a Fano threefold $V_5.$

\begin{deff}[The Fano threefold $V_5$]\label{DEF - variete Fano quintique}
A Fano threefold $V_5$ is a smooth projective threefold with Picard group isomorphic to $\eZ$ spanned by a very ample invertible sheaf $\cL,$ index 2 and degree $5.$ It verifies $\cL^{\otimes 2} \simeq \omega_{V_5}^\vee$ and $\deg c_1(\cL)^3 = 5.$
\end{deff}

\begin{prop}[\cite{FN89}, Lemma 2.2]\label{PROP - plongement V5 dans P6}
Take $V_5$ a Fano threefold and $\cL$ as in Definition \ref{DEF - variete Fano quintique}. The global sections of $\cL$ determine a closed embedding $V_5 \hookrightarrow \eP^6.$
\end{prop}

\begin{prop}\label{PROP - plongement dP dans V5 dans P9}
A smooth del Pezzo surface $S$ of degree 5 can always be described as a hyperplane section of $V_5\subset \Grass(5,2) \subset \eP^9$. These inclusions arise as the Plücker embedding of $\Grass(5,2)$ cut out by a $\eP^5$ inside a $\eP^6$ for $S$ and $V_5$ respectively.
\end{prop}

\begin{proof}
We actually do everything in reverse and start with the anti-canonical embedding of $S$ in $\eP^5$. We can use the same description as the one of \cite{BC21} for smooth del Pezzo surfaces over any field. The embedding as a codimension three variety of $\eP^5 = \eP(H^0(S,\omega_S^\vee))$ enables us to use results from \cite{BE77} and express $S$ as the zero locus of five $4\times4-$Pfaffians of an anti-symmetric matrix. 

This anti-symmetric matrix is given as the middle map $M$ in the resolution of the ideal $\cI_S$ of $S$ in $\eP^5$:
\[0\to \cG \to \cO_{\eP^5}^{\oplus 5}(-3) \overset{M}{\to} \cO_{\eP^5}^{\oplus 5}(-2) \to \cI_S \to 0.\]
In this sequence, $\cI_S$ is the Pfaffian ideal of $S$. 

The matrix $M$ is characterised by its upper triangular part, composed of 10 linear forms $(l_{ij})_{i< j}$. They form a generating family (otherwise, it would define a variety in a linear subspace of $\eP^5$) of the space of linear forms in 6 variables $H^0(S,\omega_S^\vee)$. We can extract a basis from the $l_{ij}$ and modify the remaining 4 elements to form a basis of $H^0(S,\omega_S^\vee)\oplus W$ with $W$ a $\ek-$vector space of dimension 4. Together, they give an anti-symmetric matrix $M_G$ that defines the embedding of $\Grass(5,2)$ in the Plücker embedding $\eP(H^0(S,\omega_S^\vee)\oplus W) = \eP^9$ via its $4\times 4-$Pfaffians. 

This last claim comes from the fact that the modification of $l_{ij}$ to have matrices of linear forms over $X = \eP(H^0(S,\omega_S^\vee)\oplus W)$ gives $M$ and $M_G$ as anti-symmetric matrices of
\[\cA\otimes \cO_X(-1) \to \cA^\vee \otimes \cO_X\]
with $\cA = \cO_{\eP^5}^{\oplus 5}.$ We can now identify these matrices to linear maps of
\[\bigwedge^2H^0(\eP^5,\cA) \to H^0(S,\omega_S^\vee)\oplus W\]
which are of same rank. The linear map associated to $M_G$ is invertible by construction. Moreover, up to conjugation in $\Gl_{10},$ there is only one such linear map. Taking the Grassmannian $\Grass(V,2)$ with $V$ a $5-$dimensional vector space, the anti-symmetric matrix of the Plücker coordinates in $\eP(\bigwedge^2V)$ gives the five equations of the Grassmannian as its $4\times 4-$Pfaffians. The matrix goes to the identity as a linear map and $M_G$ is defining the Grassmannian in some Plücker coordinates.

The del Pezzo surface $S$ is then exactly characterised by the intersection of this embedding of $\Grass(5,2)$ in $\eP^9$ and $\eP(H^0(S,\omega_S^\vee)).$

The same picture holds with the embedding of $V_5$ into $\eP^6$ as in Proposition \ref{PROP - plongement V5 dans P6}. To have the full sequence of inclusions $S \subset V_5 \subset \Grass(5,2),$ we do the same as before but we add an extra step. We complete the generating family of $H^0(S,\omega_S^\vee)$ into a generating family of $H^0(S,\omega_S^\vee) \oplus L$ with $L$ a $1-$dimensional vector space and then into a basis of $H^0(S,\omega_S^\vee) \oplus L \oplus W'$ with $W'$ of dimension 3 (we split $W$ into $L\oplus W'$). The restriction to the $\eP^6$ and then the $\eP^5$ gives a $V_5$ containing a del Pezzo surface.
\end{proof}

Using these embeddings, we will describe the embedding of the Hilbert scheme of lines of $S$ into the one of $V_5$ as the zero locus of a locally free sheaf.

\begin{deff}[Hilbert scheme of lines]\label{DEF - schema Hilbert droites}
The Hilbert scheme of lines of a projective variety $X,$ with respect to the very ample sheaf $\cL,$ is the scheme
\[\Hdrtes (X,\cL) = \Hilbsch[1+t,\cL]{X}\]
representing the functor $h_{1+t}^{\cL}(X)$ of coherent sheaves on $X$ such that their Hilbert polynomial (c.f. \cite[Section 5.1.4]{FGA_explained}), with respect to $\cL,$ is $1+t.$
\end{deff}

For a smooth del Pezzo surface of degree $5,$ we have $\omega_S^\vee$ as a canonical choice of polarisation to define lines.
\begin{prop}\label{PROP - inclusions des schemes de Hilbert}
With the embeddings of Proposition \ref{PROP - plongement dP dans V5 dans P9} and Proposition \ref{PROP - plongement V5 dans P6}, lines on a del Pezzo surface of degree 5 are lines of $\eP^6.$ More precisely, we have the inclusions
\[ \Hilbsch[1+2t,(\omega_S^\vee)^{\otimes 2}]{S} \subset \Hilbsch[1+2t,\omega_{V_5}^\vee]{V_5} \subset \Hilbsch[1+2t,\cO_{\eP^6}(2)]{\eP^6} = \Grass(7,2).\]
\end{prop}

This description puts us in the framework of subvarieties of the Grassmannian, which is well-known in intersection theory. However, as $S$ is not of complete intersection in $\eP^6$, we will focus on the lines on $V_5$ as an intermediate step.

\subsection{Lines on the Fano $V_5$}\label{SSEC - droites de la Fano quintique}

We will use the description of the Hilbert scheme of lines of $V_5$ given by \cite{FN89}. We will recall it in the same fashion as \cite[Section 1]{DFK25_fano22} but for a non algebraically closed field.

Take $\fC$ a curve such that $\fC_{\bar \ek} \simeq \eP^1_{\bar \ek}.$ We then have $G := \Aut(\fC)$ which is the projective linear group $\operatorname{PGL}_2(\ek)$ if $\fC$ is isomorphic to $\eP^1$ over $\Spec{\ek}$.

\begin{rmq}
Over an algebraically closed field, the automorphism group of $V_5$ is $\Aut(V_5)\simeq \PGL_2(\overline{\ek})$ (c.f. \cite[Theorem 1.1.2]{KPS18}) and there are forms of $V_5$ in general. We will then say $V_5$ for any form of it and say \emph{trivial form}\ of $V_5$ for the one with $\Aut(V_5)\simeq \PGL_2(\ek).$ In particular, the construction of the Hilbert scheme of lines on $V_5$ is based on the curve $\fC$ we just considered and if the form of $V_5$ is not trivial, we have to take non-trivial $\ek-$forms of $\eP^1$ for $\fC.$
\end{rmq}
Replacing $\eP^1$ by a $\ek-$form $\fC$ of it in \cite{DFK25_fano22}, the argument goes through without modification. We summarise the construction as follows.

\begin{prop}\label{PROP - correspondance incidence Fano quintique}We use the same notations as above and $\fC$ is either $\eP^1$ or a smooth conic without $\ek-$rational point.
\begin{enumerate}
    \item The Hilbert scheme of lines of $V_5$ is isomorphic to $\fH:= \eP(H^0(\fC,\omega_\fC^\vee)) \simeq \eP^2.$
    \item There is a locally free sheaf $\cE$ of rank 2 on $\fH$ such that $\pi : \fU := \eP(\cE) \to \fH$ is the universal family of lines.
    \item There is a morphism $\psi : \fU \to V_5$ corresponding to the projection and $\psi^\ast\cO_{V_5}(1)$ is isomorphic to $\cO_{\eP(\cE)}(1)\otimes\pi^\ast\cO_\fH(1)$.
    \item The global sections of $\cO_{\eP(\cE)}(1)\otimes\pi^\ast\cO_\fH(1)$ are isomorphic to $H^0(\fC,(\omega_\fC^\vee)^{\otimes 3}) \oplus H^0(\fC,(\omega_\fC^\vee)^{\otimes 2}).$
    \item The sheaves $\omega_\fC^\vee$ and $\cE$ have a $G-$linearised structure. Thus $\fH$ and $\fU$ have a $G$ action and $\pi$ and $\psi$ are $G-$equivariant morphisms.
\end{enumerate}
\end{prop}
\begin{rmq}\label{RMQ - morphisme de U vers P6}
Actually, the morphism $\psi$ comes from the rational $G-$equivariant map $\Psi : \fU \dashrightarrow \eP^6,$ with $\eP^6 = \eP(H^0(\fC,(\omega_\fC^\vee)^{\otimes 3})).$ The image of $\Psi$ is isomorphic to $V_5,$ thus the above claim.
\end{rmq}

The picture to keep in mind from this is the diagram
\begin{align}\label{EQU - diagramme incidence Fano}
    \xymatrix{&&\fU\ar[ld]_\psi \ar[rd]^\pi & \\ S \hspace{0.1cm}\ar@{^{(}->}[r] & V_5&&\fH.}
\end{align}

\begin{rmq}\label{RMQ - description de E}
The sheaf $\cE$ comes as $v_\ast( p_1^\ast (\omega_\fC^\vee)^{\otimes 2})$ with $p_1 : \fC \times \fC \to \fC$ the first projection and $v : \fC \times \fC \to \fH$ the double cover ramified along the diagonal. When $\fC \simeq \eP^1$ over $\Spec{\ek}$, we can choose coordinates and the double cover is defined by 
\[[e_0:e_1],[f_0:f_1] \mapsto [e_0f_0: e_0f_1 + e_1f_0 : e_1f_1].\]
\end{rmq}

\begin{rmq}
To justify that $\fH$ is isomorphic to $\eP^2$ and not a $\ek-$form of it, let us recall the two types of lines on $V_5$ as in \cite[Paragraph 1.1.1]{DFK25_fano22}. Lines on a $V_5$ are integral curves of anti-canonical degree 2 (c.f. Proposition \ref{PROP - inclusions des schemes de Hilbert}). Depending on the conormal sheaf of such a curve $\cml$ on $V_5$, there are two types of lines:
\begin{itemize}
    \item $\cml$ is of general type if the conormal sheaf is trivial,
    \item $\cml$ is of special type if the conormal sheaf is isomorphic to $\cO_\cml(-1)\oplus \cO_\cml(1).$
\end{itemize}
By \cite[Proposition 2]{DFK25_fano22}, lines of special type define a conic (the image of $\Delta(\fC)$ by the double cover $v$) on $\fH,$ which is a divisor of degree 2. We also have the anti-canonical divisor of $\fH$ which is of degree 3. The difference of these two divisor gives a divisor of degree 1 on $\fH,$ thus a $\ek-$rational point and $\fH \simeq \eP^2.$
\end{rmq}

\subsection{Lines on a del Pezzo surface}
Since we embed a degree 5 del Pezzo surface in (a $\ek-$form of) $V_5$, a line of $S$ is a line of $\eP^6$ contained in $S$, hence a line of $V_5$ contained in $S$. As $S$ is given by the zero locus of a global section of $\rest{\cO_{\eP^6}(1)}{V_5},$ its lines are given by the zero locus of the induced section of $\pi_\ast\psi^\ast\cO_{V_5}(1) \simeq \cE\otimes \cO_\fH(1)$ on $\fH$. This section is then an element of $H^0(\fH,\cE\otimes \cO_\fH(1)),$ which is isomorphic to $H^0(\fC,(\omega_\fC^\vee)^{\otimes 3}) \oplus H^0(\fC,(\omega_\fC^\vee)^{\otimes 2})$ from Proposition \ref{PROP - correspondance incidence Fano quintique}.

\begin{prop}
The Hilbert scheme of lines on a smooth del Pezzo surface of degree $5$ is a 0-dimensional subscheme of $\fH.$ It can be described as the zero locus of a section of $\cE\otimes\cO_\fH(1)$ in the projective space $\fH \simeq \eP^2.$ This sheaf is relatively orientable on $\fH.$
\end{prop}

\begin{proof}
The only thing left to prove is the relative orientability. Since $\fH$ is a projective space of dimension $2$, its Picard group is isomorphic to $\eZ$ spanned by $\cO_\fH(1)$ and $\omega_\fH \simquad \cO_\fH(1).$

On the other hand, $\det(\cE\otimes \cO_\fH(1)) \simeq \det\cE \otimes \cO_\fH(2).$ From the definition of $\cE$ as $v_\ast (p_1^\ast(\omega_\fC^\vee)^{\otimes 2})$ (c.f. Remark \ref{RMQ - description de E}), we can see it is, by a change of notations, the same as $f_\ast L^{4,0} $ with $L^{4,0} = \cO_{\eP^1}(4)\boxtimes \cO_{\eP^1}$ and $f$ the double cover of $\eP^2$ by $\eP^1\times \eP^1$ over the diagonal in \cite[Theorem 5]{Schwarzenberger}. Thus $\det \cE \simeq \cO_\fH(3)$ which is quadratically equivalent to $\omega_\fH.$
\end{proof}

\begin{rmq}\label{RMQ - classe chern fibre des droites}
The computations of the Chern classes of $\cE$ that were used in the proof (given in \cite[Theorem 5]{Schwarzenberger}) also give us $\deg c_2(\cE\otimes \cO_\fH(1)) = 10$. It gives back the 10 lines on a smooth del Pezzo surface of degree 5 over an algebraically closed field.
\end{rmq}

Up to this point, the Hilbert scheme of lines on $S$ is a $0-$dimensional smooth subvariety of length 10 in $\fH = \eP^2$ inside the Grassmannian $\Grass(7,2).$ However, the construction of the sheaf $\cE\otimes \cO_\fH(1)$ depends in general on the form of $\fC$ and the group acting on it can vary. 

\section{An oriented Schubert calculus computation}\label{SEC - calcul Schubert oriente}

\subsection{The geometrical setup and Schubert calculus}
In this section, one can assume that $\ek = \eC$ as we do some Schubert calculus in Chow groups.

The Grassmannian $\Grass(7,2)$ is a $10-$dimensional projective variety and its Chow groups can be described using Schubert calculus and Young diagrams (c.f. \cite[Chapter 14]{fulton}). In particular, $\Chowct[8]{\Grass(7,2)}$ is spanned, as a $\eZ-$module, by 
\begin{center}
    $\sigma_{4,4}=$ \ydiagram{4,4}\ and $\sigma_{5,3}\ $= \ydiagram{5,3}\ .
\end{center}

Our intersection problem comes from what is the class of $\fH = \eP^2$ in the Chow group. What we know is that, when restricted to $\Grass(6,2),$ this becomes the lines in the del Pezzo surface of degree 5. This restriction is exactly the multiplication by $c_2(\cQ)$ with $\cQ$ the universal quotient of $\Grass(7,2).$ This $c_2(\cQ)$ is exactly 
\begin{center}
    $\sigma_{1,1} =$ \ydiagram{1,1}\ .
\end{center}

\begin{rmq}
There is a convention difference here compared to \cite{fulton} and \cite{Wendt_oriented-schubert}. The tautological sub-bundle there becomes the tautological quotient sheaf here but it corresponds exactly to the same considerations.
\end{rmq}

\begin{prop}
The class of $\fH$ in $\Chowct[8]{\Grass(7,2)}$ is $10\cdot \sigma_{4,4}.$
\end{prop}

\begin{proof}
Indeed, the intersection product $\sigma_{1,1}\cdot \sigma_{4,4} = \sigma_{5,5}$ is the class of a point. But, the other product $\sigma_{1,1}\cdot \sigma_{5,3}$ is zero. Since it has to end up being 10 points, the coefficient has to be 10 at the beginning.
\end{proof}

We now need to check if and how the class $\sigma_{4,4}$ lifts to Chow-Witt groups. If it lifts, we have to compute its coefficient in $\GW(\ek)$.

\subsection{Lifting the Schubert cycle}

The extensive study of the Chow-Witt groups of Grassmannians can be found in \cite{Wendt_CHW-grass} and the interpretation of the classes in terms of oriented Schubert calculus was done in \cite{Wendt_oriented-schubert}. That is what we will use in this somehow very simple case. Let us briefly introduce the objects we need from there:

The doubling map of Young diagrams
\[\gamma : \left|\begin{array}{ccc}
    \Chowct[\bullet]{\Grass(3,1)} & \to & \Chowct[4\bullet]{\Grass(7,2)} \\
    \square & \mapsto & \boxplus
\end{array}\right.\]
gives rise to a map (c.f. \cite[Remark 3.2]{Wendt_oriented-schubert})
\[\tilde{\omega} : \Chowct[\bullet]{\Grass(3,1)} \to H^{4\bullet}(\Grass(7,2),\fId^{4\bullet}).\]
More precisely, the elements in the image of $\gamma$ (in particular our $\sigma_{4,4}$) lift to elements in the Chow-Witt group and coincide with a basis of the torsion-free part of $H^{4\bullet}(\Grass(7,2),\fId^{4\bullet}).$

The Schubert calculus can then be adapted in the groups $H^\ast(\Grass(n,k),\fId^\ast(\cL))$ and the torsion-free elements that come from the above maps are denoted
\[\fS_{\underline{a}} := \tilde{\omega}(\sigma_{\underline{a}}).\]
\begin{prop}From the above, we get the lifting of the classes we need:
\begin{itemize}
    \item The class $\sigma_{4,4}$ in $\Chowct[8]{\Grass(7,2)}$ lifts to $\Chowtilde[8]{\Grass(7,2),\cO}$.
    \item The class $c_2(\cQ)$ in $\Chowct[2]{\Grass(7,2)}$ lifts to $\Chowtilde[2]{\Grass(7,2),\det \cQ}$ as the Euler class $e^{CW}(\cQ).$
\end{itemize}
\end{prop}

\begin{rmq}
The class $c_2(\cQ) = \sigma_{1,1}$ also lifts to the $\fId-$cohomology since it is an Euler class. However, in the case of $\sigma_{4,4},$ there is no non-ambiguous description of its lift in the Chow-Witt group as the oriented Schubert calculus takes place in the $\fId-$cohomology.
\end{rmq}

Let us recall that from Proposition \ref{PROP - orientabilite grassmann}, we have the quadratic equivalence $\det \cQ \simquad \omega_{\Grass(7,2)}$ so, the twist we want for the problem to be oriented is the not trivial one.

\begin{cor}
The isomorphism $H^{10}(\Grass(7,2),\fId^{10}(\det \cQ)) \isomr \Witt(\ek)$ sends $\fS_{2}\cdot e(\cQ)$ to $\qscal{1}$.
\end{cor}

\begin{proof}
In maximal codimension, there is no torsion, so it is enough to check this claim on the $\bfW-$cohomology. This is then a direct consequence of \cite[Theorem 6.4]{Wendt_CHW-grass} with $e(\cQ)$ being $e_2$ and $\fS_{2}$ being $p_4^\perp$ as the image of $c_2^\perp$ by $\tilde{\omega}.$ A more pictorial way to see this is via Young diagrams directly and \cite[Proposition 3.14]{Wendt_oriented-schubert}.  
\end{proof}

In the end, the intersection product $e(\cQ) \cdot [\fH]$ is well-defined in $\Chowtilde[10]{\Grass(7,2),\det \cQ}.$ Its quadratic degree is a quadratically enriched count for the lines in a smooth del Pezzo surface of degree 5 and does not depend on which $\ek-$form of $S$ we work with.

\section{Witt degree using Bott residue formula}\label{SEC - calcul Witt schema Hilbert}

\subsection{The $N-$action over the trivial $\ek-$form of $V_5$}\label{SSEC - resultats prelim}

We can then compute the Witt degree under the assumption that $S$ embeds in a $V_5$ such that the curve $\fC$ from Section \ref{SSEC - droites de la Fano quintique} is isomorphic to $\eP^1_\ek$ without loss of generality. This way, we can implement the Bott residue formula from \cite{LevineAB24} since we have a $N-$action.

We will now give details about the induced $N-$action on $\fC \simeq \eP(V)$. The automorphism group $G$ is then $\PGL_2(\ek).$ It induces a $\SL_2(\ek)-$action on $\fC$ and by Proposition \ref{PROP - correspondance incidence Fano quintique}, all objects have an induced $G-$action thus an $\SL_2(\ek)-$action and a $N-$action. We can explicitly describe the $N-$action on all the objects we need.

$N$ acts on $\eP(V)$ with $V$ an $\ek-$vector space with base $e_0,e_1$ by 
\begin{align*}
    t\cdot (e_0,e_1) = &\ (t^a e_0,t^{-a}e_1)\\
    \sigma\cdot (e_0,e_1) = &\ (e_1,-e_0)
\end{align*}
with $a$ odd prime.\\
If we note $x_0,x_1$ the dual basis to the $e_i$, we get an induced action on $V^\vee$. The action of $N$ is given by the inverse of the transpose if we look at it from a matrix point of view. This gives the following:
\begin{align*}
    t\cdot (x_0,x_1) = &\ (t^{-a} x_0,t^{a}x_1)\\
    \sigma\cdot (x_0,x_1) =&\ (x_1,-x_0)
\end{align*}
with the same $a$ as before.

\begin{rmq}\label{RMQ - point fixe Sym2}
The action on $\eP(V)$ gives an action on $\eP(H^0(\eP(V),\cO_{\eP(V)}(2))) = \eP(\Symalg[2]{V})$ by
\begin{align*}
    t\cdot (e_0^2,e_0e_1,e_1^2) = &\ (t^{2a} e_0^2,e_0e_1, t^{-2a}e_1^2)\\
    \sigma\cdot (e_0^2,e_0e_1,e_1^2) = &\ (e_1^2,-e_0e_1,e_0^2).
\end{align*}
On this $\eP^2,$ the fixed point is $[0:1:0]$ if we take $\accol{e_0^2, e_0e_1,e_1^2}$ as the basis of $\Symalg[2]{V}$.
\end{rmq}

\subsection{Equivariant Euler class}\label{SSEC - classe euler equiv}

Let us note $p = [0:1:0]$ the fixed point of $\eP^2$ in the coordinates $[e_0f_0: e_0f_1 + e_1f_0 : e_1f_1]$ as described in Remark \ref{RMQ - description de E}. It is consistent with the previous description so there is no ambiguity.

From the proof of \cite[Proposition 2]{DFK25_fano22}, we have an explicit basis of $H^0(\fC,(\omega_\fC^\vee)^{\otimes 3})$ given by 
\[v_0 = e_0^5 f_0 ,\ v_1 = e^4_0 e_1 f_0 + \frac{1}{5} e^5_0 f_1 ,\ v_2 = e^3_0 e_1^2 f_0 + \frac{1}{2} e^4_0 e_1 f_1 ,\ v_3 = e^2_0 e_1^3 f_0 + e^3_0 e_1^2 f_1\]
\[v_4 = \frac{1}{2} e_0 e_1^4 f_0 + e^2_0 e_1^3 f_1 ,\ v_5 = \frac{1}{5} e_1^5 f_0 + e_0 e_1^4 f_1 ,\ v_6 = e_1^5 f_1\]
with $[e_0:e_1]$ and $[f_0:f_1]$ the coordinates on the first and second factors of $\fC\times \fC.$ The $N-$module $H^0(\fC,(\omega_\fC^\vee)^{\otimes 3})$ is the irreducible sub$-N-$module of $H^0(\fH,\cE\otimes\cO_\fH(1))$ that defines the morphism to $\eP^6$ as stated in Remark \ref{RMQ - morphisme de U vers P6} and that spans $\cE\otimes \cO_\fH(1)$.

It follows that the $N-$invariant space of $H^0(p,\rest{(\cE\otimes\cO_\fH(1))}{p})$ is
\[\qscal{e_0^5f_1, e_1^5f_0}.\]

The weights from the action by $t$ are $4a,-4a$ (there is no order question since it is even) and the action by $\sigma$ sends each generator to the opposite of the other generator. The induced representation is then $\Tilde{\cO}^-(4a)$ as defined in Definition \ref{DEF - fibre assoc N rpz}. From Theorem \ref{THO - classe euler N rpz}, we get
\[e_N^\cW(\rest{\cE\otimes\cO_\fH(1)}{p}) = -2a\Tilde{e}.\]
To finish the computation of the Bott residue formula, we have to multiply this class by the inverse of the Euler class of the tangent sheaf. We don't need $\cE$ anymore, so we can come back to the more straightforward description of the $N-$action on $\fH$ as in Remark \ref{RMQ - point fixe Sym2}. The dual of the Euler sequence for $\eP(\Symalg[2]{V}) = \fH$ gives 
\[0\to \cO \to \Symalg[2]{V^\vee}\otimes \cO_\fH(1) \to \cT_{\fH}\to 0\]
hence the global sections
\begin{align*}
    H^0(\fH,\cT_\fH) &\simeq H^0(\eP(\Symalg[2]{V}),\Symalg[2]{V^\vee}\otimes\cO_{\eP(\Symalg[2]{V})}(1))/H^0(\fH,\cO_\fH)\\
    & \simeq (\Symalg[2]{V^\vee}\otimes \Symalg[2]{V})/\ek.
\end{align*}
Taking the fibre at the point $p,$ we only have the generator $e_0e_1$ left in $\rest{\Symalg[2]{V}}{p}$. Thus, the irreducible $N-$representation of $H^0(p,\rest{\cT_\fH}{p})$ is 
\[\qscal{x_1^2, x_0^2}.\]
It has weight $2a$ and the action by $\sigma$ does not change the sign. Thus, the assigned bundle is $\tilde{\cO}^+(2a)$ over $BN.$ From Theorem \ref{THO - classe euler N rpz}, we get
\[e^\cW_N(\rest{\cT_\fH}{p}) = -a\Tilde{e}.\]

\begin{tho}\label{THO - classe euler dP 5}
Let $S$ be a smooth degree 5 del Pezzo surface. Its Hilbert scheme of lines is the zero locus of a section of $\cF = \cE\otimes \cO_\fH(1)$ in $\fH.$ Its class in the Chow-Witt group $\Chowtilde[10]{\Grass(7,2),\det \cQ}$ is also the intersection product $e(\cQ)\cdot [\fH].$ The Chow-Witt degree of this class is 
\[\deg^{CW} e(\cQ)\cdot [\fH] = 2 + 4 h \in \GW(\ek)\]
and is a quadratically enriched count of lines on $S.$
\end{tho}

\begin{proof}
Everything is mostly done. We can finish the computation of the degree of $e^{CW}(\cF)$ when we have a $N-$action. The equivariant Euler class is given by Proposition \ref{PROP - residu de Bott specialise points fixes} and the above computations. We have 
\[\deg^{W} e_N^\cW(\cF) = \frac{-2a\tilde{e}}{-a\tilde{e}} = 2\]
in the $N-$equivariant Borel-Moore homology $H^{BM}_{N,0}(\fH,\cW)[(P e^{gen})^{-1}].$ Its degree is then $2$ in $H^0(BN,\cW)$ and since we know that the degree of the Chern class $c_2(\cF)$ is 10 by Remark \ref{RMQ - classe chern fibre des droites}, we have the whole degree by Proposition \ref{PROP - classe Euler orientable}.

As it coincides with the degree of the intersection product for the trivial form and the product does not depends on the forms, we have the general result.
\end{proof}

\begin{rmq}
The statement is true over any perfect field of characteristic not 2 as the whole construction does not depend on the field. In fact, the construction of the Hilbert scheme of lines on $S$ in $\eP^2$ and in $\Grass(7,2)$ is still valid over $\eZ.$ The result we give then comes from the one over $\eZ$ and lives in $\GW(\eZ)$. The Chow-Witt degree is then only composed of 1s and $\qscal{-1}$s over any field, even though $\GW(\ek)$ may contain a lot more generators.
\end{rmq}

\bibliographystyle{alpha}
\bibliography{biblio_memoire.bib}

\end{document}